\documentclass[a4paper,11pt]{amsart}

\usepackage{amssymb}
\usepackage{amsmath}
\usepackage{amsthm}
\usepackage[all]{xy}

\newcommand{\ii}{{\boldsymbol{i}}}
\newcommand{\half}{\tfrac12}
\newcommand{\ihalf}{\tfrac\ii2}
\newcommand{\sqrthalf}{\tfrac{\sqrt3}2}

\newcommand{\CC}{\mathbb{C}}
\newcommand{\RR}{\mathbb{R}}
\newcommand{\ZZ}{\mathbb{Z}}

\newcommand{\Lt}{\Lambda}
\newcommand{\Ltt}{\widetilde\Lt}
\newcommand{\X}{\mathsf{X}}
\newcommand{\Xv}{\X^\vee}
\newcommand{\Q}{\mathsf{Q}}
\newcommand{\Qv}{\Q^\vee}
\newcommand{\Qvt}{\widetilde\Q^\vee}

\newcommand{\g}{\mathfrak{g}}
\newcommand{\h}{\mathfrak{h}}

\newcommand{\ttt}{\mathfrak{t}}
\newcommand{\z}{\mathfrak{z}}

\newcommand{\Hom}{\operatorname{Hom}}
\newcommand{\Ker}{\operatorname{Ker}}
\newcommand{\Gal}{\operatorname{Gal}}
\renewcommand{\Re}{\operatorname{Re}}
\renewcommand{\Im}{\operatorname{Im}}

\newcommand{\GL}{\mathrm{GL}}

\newcommand{\Gt}{\widetilde{G}}

\newcommand{\aff}{\text{\upshape aff}}
\newcommand{\ant}{\text{\upshape ant}}
\newcommand{\uni}{\text{\upshape uni}}
\newcommand{\red}{\text{\upshape red}}
\newcommand{\sms}{{\rm ss}}
\newcommand{\ssc}{{\rm sc}}
\newcommand{\spl}{\text{\upshape s}}
\newcommand{\cmp}{\text{\upshape c}}

\newcommand{\Ho}{\operatorname{H}}
\newcommand{\Zl}{\operatorname{Z}}
\newcommand{\Bd}{\operatorname{B}}

\newcommand{\Exp}{\mathcal{E}}
\newcommand{\Expt}{\widetilde{\mathcal{E}}}

\newcommand{\atimes}{\mathbin{\smash[b]{\underset{\text{\raisebox{2ex}{$\smash{\centerdot}$}}}{\times}}}}

\newtheorem{theorem}{Theorem}
\newtheorem{lemma}[theorem]{Lemma}
\newtheorem{corollary}[theorem]{Corollary}

\title[Component group of a real algebraic group]
{Computation of the component group \\ of an arbitrary real algebraic group}

\author{Dmitry A. Timashev}

\address{Lomonosov Moscow State University, Faculty of Mechanics and
Mathematics, Department of Higher Algebra, 119991 Moscow, Russia}
\email{timashev@mccme.ru}

\thanks{This work was supported by the Ministry of Education and Science of the Russian Federation
in the framework of the program of the Moscow Center for Fundamental and Applied Mathematics (agreement 075-15-2019-1621).}



\dedicatory{To the memory of Aleksander Vasilyevich Mikhalev}

\begin{document}

\date{November 9, 2023}

\begin{abstract}
We compute explicitly the group of connected components $\pi_0G(\RR)$ of the real Lie group $G(\RR)$ for an arbitrary (not necessarily linear) connected algebraic group $G$ defined over the field $\RR$ of real numbers. In particular, it turns out that $\pi_0G(\RR)$ is always an elementary Abelian 2-group. The result looks most transparent in the cases where $G$ is a linear algebraic group or an Abelian variety. The computation is based on structure results on algebraic groups and Galois cohomology methods.
\end{abstract}

\maketitle

\section*{Introduction}

Let $G$ be an arbitrary algebraic group defined over the field $\RR$ of real numbers. If $G$ is connected (in the Zariski topology), then its complex locus $G(\CC)$ is also a connected (in the classical Hausdorff topology) complex Lie group, while the real locus $G(\RR)$ is a real Lie group, which is no longer connected in general: $G=\GL_n$ may serve as an example. It is natural to address a question about the structure of the group $\pi_0G(\RR)=G(\RR)/G(\RR)^\circ$ of connected components of~$G(\RR)$, where $G(\RR)^\circ$ denotes the identity component.

First results in this direction go back to \'E.~Cartan: if $G$ is semisimple and simply connected, then $G(\RR)$ is connected, i.e., the component group is trivial; see, e.g., \cite[Cor.\,4.7]{grp.red.add}. H.~Matsumoto proved in 1964 that, $\pi_0G(\RR)$ is an elementary Abelian 2-group whenever $G$ is a linear algebraic group; see \cite[Thm.\,1.2]{Matsumoto}. Recently M.~Borovoi and the author obtained a combinatorial description of the component group for a linear algebraic group $G$ as the stabilizer of a point for an explicitly given action of a certain finite group on a certain finite set, which allows one to compute the group $\pi_0G(\RR)$ in any given case; see \cite[Thm.\,9.2]{Galois:red}.

A weak point in this description is that it provides a method for computing the component group but does not lead to an explicit formula for the latter group. This defect was corrected by the author in~\cite{pi0:lin}, where an explicit formula for $\pi_0G(\RR)$ was obtained representing this group as a quotient of the cocharacter lattice of a maximal split torus in~$G$; see also~\cite{pi0:red}. In this paper, a similar formula is obtained for an arbitrary (not necessarily linear) connected algebraic group. This gives a final answer to the above question about the component group.

Here is the structure of the paper. We start by providing necessary material on algebraic groups (in Section~\ref{s:alg.grp.str}) and Galois cohomology (in Section~\ref{s:Galois}). In Section~\ref{s:pi0}, we prove Theorem~\ref{t:main}, which is the main result of the paper. In Section~\ref{s:example}, we illustrate our formula for the component group by an example of the real locus of an elliptic curve.

\section{Structure of connected algebraic groups}
\label{s:alg.grp.str}

\subsection{}

We fix a square root of $-1$ in the field $\CC$ of complex numbers and denote it by~$\ii$.

We consider algebraic groups defined over~$\RR$. We shall identify such an algebraic group $G$ with its complex locus $G(\CC)$ equipped with an anti-regular involutive group automorphism~$\sigma$ (the \emph{real structure}) defining the action of the Galois group $\Gal(\CC/\RR)$ on~$G(\CC)$. (The name ``anti-regular'' means that $\sigma$ sends regular functions on Zariski open subsets in $G$ to functions that are complex conjugate to regular ones.) In this setting, $G(\RR)=G^\sigma$ is the fixed point subgroup of~$\sigma$.

For an arbitrary algebraic or Lie group~$H$, denote by $H^\circ$ the connected component of identity in~$H$ and by $\h$ (the same lowercase German letter) the Lie algebra of~$H$.

We say that an algebraic or Lie group $H$ decomposes into an \emph{almost direct product} of subgroups $H_1$ and $H_2$, and denote it by
\begin{equation*}
  H=H_1\atimes H_2,
\end{equation*}
if $H={H_1\cdot H_2}$, the subgroups $H_1$ and $H_2$ commute with each other, and the intersection ${H_1\cap H_2}$ is finite (but not necessarily trivial).

\subsection{}

We shall need some general structural results on connected algebraic groups. We use a survey paper of M.~Brion \cite{alg.grp} as a source.

The most widely known are \emph{affine} algebraic groups, i.e., those which are affine varieties, from the algebro-geometric viewpoint. Each affine algebraic group is \emph{linear}, i.e., admits a faithful linear representation. The converse is also true: a linear algebraic group is isomorphic to a Zariski closed subgroup in~$\GL_n$ (for some~$n$) and therefore is an affine variety, whose coordinate algebra is generated by the matrix entries of the faithful linear representation.

An opposite class of algebraic groups is constituted by \emph{anti-affine} groups, which are the algebraic groups having no non-constant regular functions. Obviously, anti-affine algebraic groups are connected. Examples of anti-affine groups are given by Abelian varieties, i.e., those algebraic groups which are irreducible projective varieties. However the class of anti-affine groups is wider; see \cite[5.5]{alg.grp}.

For arbitrary connected algebraic groups, there is the following \emph{Rosenlicht decomposition}:
\begin{theorem}[{\cite[5.1.1]{alg.grp}}]
\label{t:Rosenlicht}
Let $G$ be a connected algebraic group. There exist the largest connected affine subgroup $G_\aff$ and the largest (connected) anti-affine subgroup $G_\ant$ in~$G$. The subgroup $G_\aff$ is normal and $G_\ant$ is central in~$G$, and
\begin{equation}
\label{e:Rosenlicht}
G=G_\aff\cdot G_\ant.
\end{equation}
\end{theorem}

The structure of affine (i.e., linear) algebraic groups is well known. There is a \emph{Levi decomposition} \cite[Thm.\,VIII.4.3]{alg.Lie}:
\begin{equation}\label{e:Levi}
  G_\aff=G_\uni\rtimes G_\red,
\end{equation}
where $G_\uni$ is the largest normal unipotent subgroup (\emph{unipotent radical}) of~$G$ and $G_\red$ is a maximal connected reductive subgroup (\emph{Levi subgroup}). A Levi subgroup is not unique, but all Levi subgroups are conjugate to each other. There is an almost direct product decomposition:
\begin{equation}\label{e:red}
  G_\red=G_\sms\atimes S,
\end{equation}
where $G_\sms=[G_\red,G_\red]$ is a connected semisimple group and $S=Z(G_\red)^\circ$ is an algebraic torus.

The structure of anti-affine algebraic groups is described in \cite[5.5]{alg.grp}. There is an exact sequence
\begin{equation*}
1 \longrightarrow (G_\ant)_\aff = (G_\ant\cap G_\aff)^\circ \longrightarrow G_\ant \longrightarrow A \longrightarrow 1,
\end{equation*}
where $A$ is an Abelian variety and $(G_\ant)_\aff$ is a connected commutative linear algebraic group. There is a decomposition
\begin{equation*}
  (G_\ant)_\aff=S_0\times V,
\end{equation*}
where $S_0$ is an algebraic torus and $V$ is a commutative unipotent group, i.e., the additive group of a vector space (a \emph{vector group}). Theorem~\ref{t:Rosenlicht} and the decompositions \eqref{e:Levi} and \eqref{e:red} imply inclusions $S_0\subseteq S$ and $V\subseteq G_\uni$.

The decompositions \eqref{e:Rosenlicht}--\eqref{e:red} imply that $Z=S\cdot G_\ant$ is a connected commutative algebraic subgroup in~$G$ and
\begin{equation}\label{e:gen.Levi}
G=G_\uni\cdot\bigl(G_\sms\atimes Z\bigr).
\end{equation}

\subsection{}

Let $T$ be an algebraic torus of dimension~$n$ defined over~$\RR$. The torus $T$ is called \emph{split} if there exists a coordinate system (i.e., an isomorphism $T(\CC)\simeq\CC^\times\times\dots\times\CC^\times$) in which the real structure is given by the formula
\begin{equation*}
  \sigma(t_1,\dots,t_n)=(\bar{t}_1,\dots,\bar{t}_n),
\end{equation*}
and $T$ is called \emph{anisotropic} if there exists a coordinate system in which
\begin{equation*}
  \sigma(t_1,\dots,t_n)=(\bar{t}_1^{-1},\dots,\bar{t}_n^{-1}).
\end{equation*}
(The bar here denotes complex conjugation.) If $T$ is split, then ${T(\RR)\simeq(\RR^\times)^n}$. If $T$ is anisotropic, then $T(\RR)\simeq(U_1)^n$ is a compact real torus; here $U_1$ denotes the unit circle on the complex plane. In general, there exist the largest split subtorus $T_\spl\subseteq T$ and the largest anisotropic subtorus $T_\cmp\subseteq T$, so that there is an almost direct product decomposition $T=T_\spl\atimes T_\cmp$.

Let $\Xv=\Xv(T)=\Hom(\CC^\times,T)$ be the cocharacter lattice of~$T$. By identifying cocharacters with their differentials at the unity, we may (and will) regard $\Xv$ as a lattice in~$\ttt$. There is an exact sequence
\begin{equation*}
  0 \longrightarrow \ii\Xv \longrightarrow \ttt \overset{\Exp}\longrightarrow T \longrightarrow 1,
\end{equation*}
where $\Exp$ is the scaled exponential map given by the formula $\Exp(x)=\exp(2\pi{x})$. This map can be regarded as the universal covering of the complex Lie group~$T(\CC)$.

The involution $\sigma$ acts on~$\ttt$ preserving the lattice~$\Xv$. The sublattices $\Xv_+$ and~$\Xv_-$ consisting of cocharacters preserved and multiplied by~$-1$ under the action of~$\sigma$ are just the cocharacter lattices of $T_\spl$ and~$T_\cmp$, respectively.

\section{Galois cohomology}
\label{s:Galois}

As in our previous paper~\cite{pi0:lin}, we use techniques of Galois cohomology over real numbers for computation of the component group. For the reader's convenience, we recall here briefly necessary material on Galois cohomology by following \cite[\S2]{pi0:lin}. We refer the reader to the classical book of J.-P.~Serre \cite{Galois} for a detailed exposition.

\subsection{}

Let $A$ be an arbitrary group on which the Galois group $\Gal(\CC/\RR)$ acts by group automorphisms. Such an action is defined by an involutive automorphism $\sigma$ of~$A$.\showhyphens{automorphism}

An element $z\in A$ is called a \emph{$1$-cocycle} with coefficients in $A$ if $\sigma(z)=z^{-1}$. The group $A$ acts by twisted conjugation on the set $\Zl^1(\RR,A)$ of 1-cocycles:
\begin{equation*}
z\overset{a}\longmapsto a\cdot z\cdot\sigma(a)^{-1}.
\end{equation*}
The orbit $[z]$ of the cocycle $z$ is called its \emph{cohomology class} and the orbit set
\begin{equation*}
\Ho^1(\RR,A)=\Zl^1(\RR,A)/A
\end{equation*}
is called the \emph{first cohomology set} of $\Gal(\CC/\RR)$ with coefficients in~$A$ or, for short, the first Galois cohomology set of~$A$. This set contains a base point, which is the class $[1]$ of the trivial cocycle, also denoted by $\Bd^1(\RR,A)$ and called the set of \emph{$1$-coboundaries}. In case $A$ is Abelian the sets $\Zl^1(\RR,A)$, $\Bd^1(\RR,A)$, and $\Ho^1(\RR,A)=\Zl^1(\RR,A)/\Bd^1(\RR,A)$ are Abelian groups.

Any short exact sequence of groups
\begin{equation}\label{e:short}
1 \longrightarrow A \longrightarrow B \longrightarrow C \longrightarrow 1
\end{equation}
on which $\Gal(\CC/\RR)$ acts in a compatible way gives rise to a long exact cohomology sequence
\begin{equation}\label{e:long.cohom}
1 \longrightarrow A^\sigma \longrightarrow B^\sigma \longrightarrow C^\sigma \longrightarrow
\Ho^1(\RR,A) \longrightarrow \Ho^1(\RR,B) \longrightarrow \Ho^1(\RR,C).
\end{equation}
The maps in the sequence \eqref{e:long.cohom} are obvious, except for the fourth one from the left, which is given by the formula
\begin{equation*}\label{e:connect}
c\longmapsto[b^{-1}\sigma(b)],\qquad\forall c\in C^\sigma,
\end{equation*}
where $b\in B$ is an arbitrary preimage of~$c$, so that $b^{-1}\sigma(b)\in A$. Exactness is understood in the sense of maps of pointed sets.

In the case where the exact sequence \eqref{e:short} splits, i.e., $B=A\times C$, the picture simplifies: it is easy to see that
\begin{equation*}
  \Ho^1(\RR,B)\simeq\Ho^1(\RR,A)\times\Ho^1(\RR,C).
\end{equation*}

\subsection{}

If $U$ is a unipotent linear algebraic group defined over~$\RR$, then $\Ho^1(\RR,U)=\{[1]\}$. An elementary proof of this well-known fact (see \cite[Lemma~6.2(ii)]{Galois:red}) is based on an observation that any element in a unipotent group possesses a unique square root (this follows from bijectivity of the exponential map for unipotent groups). Specifically, each $z\in\Ho^1(\RR,U)$ is represented as $z=z^{1/2}\cdot\sigma(z^{-1/2})$ (since $z^{1/2}$ is also a cocycle), i.e., is a coboundary.

\section{Component group}
\label{s:pi0}

Let $G$ be a connected algebraic group defined over~$\RR$. Our goal is to compute the group of connected components $\pi_0G(\RR)=G(\RR)/G(\RR)^\circ$ of the real locus~$G(\RR)$. We use the notation from Section~\ref{s:alg.grp.str}.

\subsection{}

The decomposition \eqref{e:gen.Levi} extends to real loci.

\begin{lemma}
\label{l:gen.Levi(R)}
$G(\RR)=G_\uni(\RR)\cdot\bigl(G_\sms\atimes Z\bigr)(\RR)$.
\end{lemma}
\begin{proof}
Take any $g\in G(\RR)$. We can decompose it as $g=u\cdot h$, where $u\in G_\uni(\CC)$, $h\in\bigl(G_\sms\atimes Z\bigr)(\CC)$, so that $g=\sigma(g)=\sigma(u)\cdot\sigma(h)$. Hence $v:=u^{-1}\cdot\sigma(u)=h\cdot\sigma(h)^{-1}\in V(\CC)$ and $\sigma(v)=v^{-1}$, i.e., $v$~is a Galois 1-cocycle. Since $V$ is a vector group, there exists a unique square root $v^{1/2}\in V(\CC)$, which is also a 1-cocycle, i.e., $\sigma(v^{1/2})=v^{-1/2}$. By replacing $u$ with $u\cdot v^{1/2}$ and $h$ with $v^{-1/2}\cdot h$, we get $u\in G_\uni(\RR)$, $h\in\bigl(G_\sms\atimes Z\bigr)(\RR)$.
\end{proof}

The exponential map $\exp:\g_\uni\to G_\uni$ is an isomorphism of algebraic varieties over $\RR$ and therefore yields a diffeomorphism between $\g_\uni(\RR)$ and~$G_\uni(\RR)$. Thus, $G_\uni(\RR)$ is a connected Lie group and Lemma~\ref{l:gen.Levi(R)} implies a canonical isomorphism
\begin{equation*}
\pi_0G(\RR)\simeq\pi_0(G_\sms\atimes Z)(\RR).
\end{equation*}
For this reason, we may (and will) assume in the sequel without loss of generality that $G=G_\sms\atimes Z$, where $G_\sms=[G,G]$ is semisimple, and $Z=Z(G)^\circ$ is the connected center of~$G$.

\subsection{}

We proceed with the same ideas as in~\cite{pi0:lin}. Let $G_\ssc$ be a simply connected semisimple algebraic group which covers~$G_\sms$. Consider a group $\Gt=G_\ssc\times\z$, where $\z$ is regarded as a vector group. The Lie group $\Gt(\CC)$ is simply connected and there is an exact sequence
\begin{equation}\label{e:uni.cover}
1 \longrightarrow \Gamma \longrightarrow \Gt=G_\ssc\times\z \overset\tau\longrightarrow G=G_\sms\atimes Z \longrightarrow 1.
\end{equation}
The universal covering map $\tau$ is given by the formula
\begin{equation*}
\tau(\tilde{g})=\tau_\ssc(g_\ssc)\cdot\Exp(z),\qquad\forall\;\tilde{g}=(g_\ssc,z),\ g_\ssc\in G_\ssc,\ z\in\z,
\end{equation*}
where $\tau_\ssc:G_\ssc\to G_\sms$ is the universal covering of the derived subgroup of~$G$ and $\Exp:\z\to Z$ is the scaled exponential map given by the formula
\begin{equation}\label{e:Exp}
\Exp(z)=\exp(2\pi{z}).
\end{equation}
The group $\Gamma=\Ker(\tau)$ is a discrete central subgroup in $\Gt(\CC)$ isomorphic to the fundamental group of~$G(\CC)$. It is worth noting that in general $\tau$ is a homomorphism of complex Lie groups, not of algebraic groups.

Let $T_\sms$ be a maximal torus in $G_\sms$ defined over~$\RR$. Then $T_\ssc=\tau_\ssc^{-1}(T_\sms)$ is a maximal torus in~$G_\ssc$ and $\Xv(T_\ssc)=\Qv$ is the coroot lattice of~$G_\sms$. There is a commutative diagram with exact rows and columns:
\begin{equation*}
\xymatrix{
         &               &  0                                            &            0                         &   \\
1 \ar[r] & \Gamma \ar[r] &    T_\ssc\times\z \ar[r]^-\tau \ar[u]         &   T_\sms\atimes Z \ar[r] \ar[u]      & 1 \\
         &    0   \ar[r] & \ttt_\ssc\oplus\z \ar[r]^-\sim \ar[u]^\Expt & \ttt_\sms\oplus\z \ar[r] \ar[u]^\Exp & 0 \\
         &    0   \ar[r] &            \ii\Qv \ar[r] \ar[u]               &         \ii\Lt    \ar[u]             &   \\
         &               &               0   \ar[u]                      &            0      \ar[u]             &   \\
}\end{equation*}
Here $\Exp$ and $\Expt$ are the scaled exponential maps for the respective connected commutative Lie groups, given by the formul{\ae} similar to~\eqref{e:Exp}, and $\Expt(x,z)=(\Exp_\ssc(x),z)$ for any $x\in\ttt_\ssc$, $z\in\z$, where $\Exp_\ssc:\ttt_\ssc\to T_\ssc$ is also the scaled exponential map. The lattice $\Lt=\ii\cdot\Ker(\Exp)$ is preserved by the action of~$\sigma$ and
\begin{equation*}
\Gamma=\Expt(\ii\Lt)\simeq\ii\Lt\,/\,\ii\Qv.
\end{equation*}

\subsection{}

As in \cite[3.2]{pi0:lin}, we write down a fragment of the exact Galois cohomology sequence \eqref{e:long.cohom} related to the exact sequence of groups~\eqref{e:uni.cover}:
\begin{equation*}
\cdots \longrightarrow \Gt(\RR) \longrightarrow G(\RR) \longrightarrow \Ho^1(\RR,\Gamma) \longrightarrow \Ho^1(\RR,\Gt) \longrightarrow \cdots
\end{equation*}
The Lie group $\Gt(\RR)=G_\ssc(\RR)\times\z(\RR)$ is connected, because both factors are connected. Hence the covering $\tau$ maps $\Gt(\RR)$ onto~$G(\RR)^\circ$. We also have
\begin{equation*}\label{e:H1}
\Ho^1(\RR,\Gt) \simeq \Ho^1(\RR,G_\ssc) \times \Ho^1(\RR,\z) \simeq \Ho^1(\RR,G_\ssc)
\end{equation*}
by vanishing of the Galois cohomology of the vector group~$\z$. We thus arrive at an exact sequence
\begin{equation}\label{e:exact}
1 \longrightarrow \pi_0G(\RR) \overset\delta\longrightarrow \Ho^1(\RR,\ii\Lt\,/\,\ii\Qv) \overset{\iota}\longrightarrow \Ho^1(\RR,G_\ssc).
\end{equation}
The map $\delta$ sends a connected component $gG(\RR)^\circ$ of $G(\RR)$ to the cohomology class $[\ii\lambda+\ii\Qv]$ such that $\Expt(\ii\lambda)=\tilde{g}^{-1}\sigma(\tilde{g})\in\Zl^1(\RR,\Gamma)$, where $\tilde{g}\in\Gt$ is such that $\tau(\tilde{g})=g$. The map $\iota$ sends a cohomology class $[\ii\lambda+\ii\Qv]$ to the cohomology class $[\Exp_\ssc(\ii\lambda_\ssc)]$, where $\lambda_\ssc$ is the projection of $\lambda\in\ttt_\ssc\oplus\z$ to~$\ttt_\ssc$.

It remains to compute $\Ker(\iota)$.

\subsection{}
\label{ss:tori&lattices}

The maximal torus $T_\sms$ in $G_\sms$ decomposes into an almost direct product of a split torus and an anisotropic torus: $T_\sms=T_\spl\atimes T_\cmp$. Choose $T_\sms$ in such a way that its split part $T_\spl$ is a maximal split torus in~$G_\sms$.

The involution $\sigma$ acts linearly on the lattice~$\Lt$. The real vector space $\Lt_\RR$ spanned by this lattice decomposes into the direct sum of the eigenspaces~$\Lt_\RR^{\pm\sigma}$ on which $\sigma$ acts with eigenvalues $\pm1$, respectively. Consider the lattices $\Lt_\pm=\Lt\cap\Lt_\RR^{\pm\sigma}$, and also the lattices $\Ltt_\pm$ which are the images of $\Lt$ under the projections to $\Lt_\RR^{\pm\sigma}$ given by the formul{\ae}
\begin{equation*}
  \lambda\mapsto\lambda_\pm=\half(\lambda\pm\sigma(\lambda)),
\end{equation*}
respectively. There are obvious inclusions
\begin{equation*}
  \Lt_\pm\subseteq\Ltt_\pm\subseteq\half\Lt_\pm.
\end{equation*}
We use similar notation for any $\sigma$-stable lattice in~$\Lt_\RR$, in particular, for~$\Qv$.

\subsection{}

Here is the main result of this paper.
\begin{theorem}\label{t:main}
Let $G$ be a connected algebraic group defined over~$\RR$, $G_\ant\subseteq G$ be the largest anti-affine subgroup, $G_\red\subseteq G$ be a maximal connected reductive subgroup, $S=Z(G_\red)^\circ$ be its connected center, $G_\sms=[G_\red,G_\red]$ be its derived subgroup, $T_\sms\subseteq G_\sms$ be a maximal torus containing a maximal split torus $T_\spl\subseteq G_\sms$, and $Z=S\cdot G_\ant$. Then the group of connected components of the real locus $G(\RR)$ is computed by the following formula (in the notation of Subsection~\ref{ss:tori&lattices}):
\begin{equation*}
  \pi_0G(\RR) \simeq \Lt_+/(2\Ltt_++\Qv_+).
\end{equation*}
Here $\Lt$ is the lattice in the Lie algebra $\ttt_\sms\oplus\z$ of the algebraic group $T_\sms\atimes Z$ such that $\ii\Lt$ is the kernel of the universal covering $\Exp:\ttt_\sms\oplus\z\to T_\sms\atimes Z$, $\Exp(y)=\exp(2\pi{y})$, and $\Qv$ is the coroot lattice of $G_\sms$ with respect to~$T_\sms$. The connected component of $G(\RR)$ corresponding to the coset of $\lambda\in\Lt_+$ is represented by
\begin{equation*}
  \exp(\pi\ii\lambda)=\Exp(\ii\lambda/2).
\end{equation*}
\end{theorem}

\begin{proof}[The proof\/\nopunct] is similar to the proof of Theorem~6 in~\cite{pi0:lin}. So we skip some technical details.

We use the exact sequence~\eqref{e:exact}. Similarly to the proof of \cite[Lemma~2]{pi0:lin}, we have:
\begin{align*}
  \Zl^1(\RR,\ii\Lt\,/\,\ii\Qv) &=      \ii\Lt\cap(\ii\Ltt_++\ihalf\Qv_-)/\ii\Qv, \\
  \Bd^1(\RR,\ii\Lt\,/\,\ii\Qv) &=      (2\ii\Ltt_++\ii\Qv)/\ii\Qv,                 \\
  \Ho^1(\RR,\ii\Lt\,/\,\ii\Qv) &\simeq \Lt\cap(\Ltt_++\half\Qv_-)/(2\Ltt_++\Qv).
\end{align*}
The image of a cohomology class $[\ii\lambda+\ii\Qv]$ under the map~$\iota$ is the cocycle $\Exp_\ssc(\ii\lambda_\ssc)\in Z(G_\ssc)$. Since $Z(G_\ssc)$ has finite order~$N$, we have $N\lambda_\ssc\in\Qv\subseteq\Xv(T_\sms)\subseteq\Lt$, whence $\lambda_\ssc\in\Xv(T_\sms)_\RR\subseteq\Lt_\RR$ and $(\lambda_\ssc)_+\in\Xv(T_\sms)_\RR^\sigma=\Xv(T_\spl)_\RR$.

Similarly to \cite[Lemma~4]{pi0:lin}, $\Exp_\ssc(\ii\lambda_\ssc)$ is a coboundary in $G_\ssc$ if and only if it belongs to~$(T_\ssc)_\spl$, i.e., $\lambda_\ssc\in{\Xv(T_\spl)_\RR+\Qv}$. Since $(\lambda_\ssc)_-=\lambda_-\in\half\Qv_-$\,, the coboundary condition is equivalent to $\lambda_-\in\Qvt_-$. As $\lambda$ is defined up to the shift by a vector in~$\Qv$, we may assume without loss of generality that $\lambda_-=0$, i.e., $\lambda\in\Lt_+$.

Thus
\begin{equation*}
  \pi_0G(\RR)\simeq\Ker(\iota)\simeq\Lt_+/(2\Ltt_++\Qv).
\end{equation*}
For $\lambda\in\Lt_+$, $\tilde{g}=\Expt(-\ii\lambda/2)$ is a cocycle in~$\Gt$, whence
\begin{equation*}
  \tilde{g}^{-1}\sigma(\tilde{g})=\tilde{g}^{-2}=\Expt(\ii\lambda)
\end{equation*}
and
\begin{equation*}
  g=\tau(\tilde{g})=\Exp(-\ii\lambda/2)=\Exp(\ii\lambda/2)
\end{equation*}
represents the connected component in~$G(\RR)$ corresponding to the coset of~$\lambda$.
\end{proof}

A straightforward consequence of Theorem~\ref{t:main} generalizes Matsumoto's theorem \cite[Thm.\,1.2]{Matsumoto}:
\begin{corollary}
$\pi_0G(\RR)$ is an elementary Abelian 2-group of finite order.
\end{corollary}

\subsection{}

Theorem~\ref{t:main} takes the most transparent form in the cases where $G$ is a linear algebraic group or an Abelian variety.

The case of linear algebraic groups is examined in detail in~\cite{pi0:lin}. In this case, $Z=S$ is an algebraic torus, $\Lt=\Xv(T)$ is the cocharacter lattice of a maximal torus $T=S\atimes T_{\sms}$ in~$G$ containing a maximal split torus $T_\spl$ (here $T_\spl$ denotes a maximal split torus in~$G$, not in~$G_\sms$!), and $\Lt_+=\Xv(T_\spl)$.

In the case where $G$ is an Abelian variety, the groups $G_\aff$, $G_\red$, $G_\sms$, $G_\ssc$, $T_\sms$, $T_\ssc$, $T_\spl$, $S$ are trivial and $Z=G$. The exact sequence \eqref{e:uni.cover} reads as
\begin{equation*}\label{e:uni.cover:Abel}
  0 \longrightarrow \Lt \longrightarrow \g \overset\Exp\longrightarrow G \longrightarrow 1,
\end{equation*}
i.e., $\Lt$ is the (scaled) period lattice of the Abelian variety~$G$. In this case, $\Qv=0$ and $\pi_0G(\RR)\simeq\Lt_+/2\Ltt_+$.

\section{Example}
\label{s:example}

Examples of computation of component groups for linear algebraic groups were considered in \cite[\S4]{pi0:lin}. Here we consider the ``opposite'' case of Abelian varieties. As an example, we consider one-dimensional Abelian varieties, i.e., elliptic curves.

\subsection{}

Let $G$ be an elliptic curve defined over~$\RR$. In this case, $\g=\CC$ is the complex plane, the lattice $\Lambda\subset\CC$ has rank~2, and $G\simeq\CC/\ii\Lambda$. Since $G$ is defined over~$\RR$, the lattices $\Lambda$ and $\ii\Lambda$ are preserved by complex conjugation and, in particular, $\ii\Lambda$~intersects the real axis. After rescaling we may assume without loss of generality that $\ii\Lambda=\ZZ+\ZZ\omega$, where $-\half\le\Re(\omega)<\half$. By invariance under complex conjugation, there are only two possibilities for the fundamental period~$\omega$: either $\omega\in\ii\RR$ or $\omega\in-\half+\ii\RR$.

In the first case, $\ii\Lt_+=\ii\Ltt_+=\ZZ\omega$, hence $\pi_0G(\RR)\simeq\ZZ/2\ZZ$, i.e., $G(\RR)$~has two connected components.

In the second case, $\ii\Lt_+=\ZZ(\omega-\bar\omega)$, while $\ii\Ltt_+=\ZZ(\omega-\bar\omega)/2$. Hence $\pi_0G(\RR)=0$, i.e., $G(\RR)$~is connected.

\subsection{}

One can see that differently. Let us define $G$ in the Weierstrass normal form by the equation
\begin{equation*}
  y^2z=x^3+pxz^2+qz^3.
\end{equation*}
(Here $x,y,z$ are homogeneous coordinates on a projective plane). It is well known (see, e.g., \cite[Chap.\,I, \S6]{elliptic}) that the coefficients $p,q$ are given by Eisenstein series:
\begin{align*}
  p&=-15\sum_{\lambda\in\ii\Lt\setminus\{0\}}\frac{1}{\strut\lambda^4}\;, &
  q&=-35\sum_{\lambda\in\ii\Lt\setminus\{0\}}\frac{1}{\strut\lambda^6}\;.
\end{align*}
Consider the cubic trinomial $f(x)=x^3+px+q$ and its discriminant $D(f)=-4p^3-27q^2$.

If $\omega\in\ii\RR$, then, without loss of generality, $\Im(\omega)>0$ and $D(f)>0$.

Indeed, since $G$ is a smooth curve, $D(f)\ne0$ and therefore has constant sign on the ray $\ii\RR_{>0}$. Thus it suffices to consider a single value of $\omega\in\ii\RR_{>0}$, e.g., $\omega=\ii$.

In this case, $G$ possesses complex multiplication: the period lattice is preserved by multiplication with~$\ii$. On the other hand, when the period lattice is multiplied by~$\ii$, the series $q$ is multiplied by~$-1$. Hence $q=0$.

In this case, $\ii\Lambda=\ZZ[\ii]$ is the ring of Gaussian numbers, which is a Euclidean ring. Hence the series $p$ can be written in the form of an Euler product:
\begin{equation*}
  p=-15\cdot\bigl(1^4+(-1)^4+\ii^4+(-\ii)^4\bigr)\cdot\prod_{\rho}\frac{1}{1-1/\strut\rho^4}\;,
\end{equation*}
where the product is taken over all Gaussian primes~$\rho$ up to multiplication by invertible elements $\pm1,\pm\ii$ of~$\ZZ[\ii]$. Gaussian primes which are not integers can be grouped in pairs of mutually conjugate numbers, whence
\begin{equation*}
  p=-60\cdot\prod_{\Im(\rho)>0}\frac{1}{|1-1/\strut\rho^4|^2}\cdot\prod_{\Im(\rho)=0}\frac{1}{1-1/\strut\rho^4}<0\;.
\end{equation*}
It follows that $D(f)=-4p^3>0$.

Removing the point at infinity $g_\infty=(0:1:0)$, we can define $G(\RR)\setminus\{g_\infty\}$ by the equation $y^2=f(x)$ on the real coordinate plane. Since $f$ has three distinct real roots, it is clear that $G(\RR)$ has two connected components (see Fig.~\ref{f:2-component}).
\begin{figure}[!h]
  \centering
  \input{2-component.pic}
  \caption{}\label{f:2-component}
\end{figure}

If now $\omega\in-\half+\ii\RR$, then $D(f)<0$. The proof is similar to the previous case. Since $D(f)$ has constant sign on the ray $-\half+\ii\RR_{>0}$, we may assume without loss of generality that $\omega=-\half+\ii\sqrthalf$ is a cube root of~$1$.

In this case $G$ still possesses complex multiplication: the period lattice is preserved by multiplication with~$\omega$. However, when the period lattice is multiplied by~$\omega$, the series $p$ is multiplied by $\omega^{-1}=\bar\omega$, hence $p=0$.

Here $\ii\Lambda=\ZZ[\omega]$ is the ring of Eisenstein numbers, which is Euclidean again. Therefore, as in the previous case, the series $q$ can be written as an Euler product:
\begin{align*}
  q&=-35\cdot\bigl(1^6+\omega^6+\bar\omega^6+(-1)^6+(-\omega)^6+(-\bar\omega)^6\bigr)\cdot\prod_{\rho}\frac{1}{1-1/\strut\rho^6}\;, \\
\intertext{(the product is taken over all Eisenstein primes up to multiplication by invertible elements $\pm1,\pm\omega,\pm\bar\omega$ of~$\ZZ[\omega]$)}
   &=-210\prod_{\Im(\rho)>0}\frac{1}{|1-1/\strut\rho^6|^2}\cdot\prod_{\Im(\rho)=0}\frac{1}{1-1/\strut\rho^6}<0\;.
\end{align*}
It follows that $D(f)=-27q^2<0$.

In this case the trinomial $f$ has a single real root and therefore $G(\RR)$ is connected (see Fig.~\ref{f:1-component}).
\begin{figure}[!h]
  \centering
  \input{1-component.pic}
  \caption{}\label{f:1-component}
\end{figure}

\end{document}